\documentclass[12pt,leqno]{article}
\usepackage{amsmath,amssymb,amsfonts,amsthm,latexsym,amstext,amscd}
\usepackage[english]{babel}
\usepackage{fancyhdr}
\usepackage{graphicx}

\newtheorem{proposition}{Proposition}

\newtheorem*{lem}{Lemma} %%%% for unnumbered statements

\newtheorem*{theo}{Theorem}

\theoremstyle{definition}

\newtheorem*{rema}{Remark}

\newcommand{\beq}{\begin{equation}}
\newcommand{\eeq}{\end{equation}}

\begin{document}

\title%%[On polyhedra inscribed in $S^2$, with approximately equal edges]
{On polyhedra inscribed in $S^2$, with approximately equal edges}

\author%%[E. Makai, Jr.]
{Endre Makai, Jr.%%\thanks{}
}

\date{}

\numberwithin{equation}{section}

\maketitle

\medskip

{\centerline{Alfr\'ed R\'enyi Institute of Mathematics, ELKH}}

{\centerline{H-1364 Budapest, P.O. Box 127, Hungary}}

{\centerline{{\rm{http://www.renyi.hu/\~{}makai}}}}

{\centerline{{\tt{makai.endre@renyi.hu}}}}

\medskip

\medskip

\medskip

{\bf{Abstract.}} 
We consider triangle faced convex polyhedra inscribed in the unit sphere $S^2$
in ${\Bbb{R}}^3$.
One way of measuring their deviation from regular polyhedra with triangular
faces is to consider the quotient of the lengths of the longest and the
shortest edges. If the number of faces tends to infinity, and the polyhedron
with this number of faces varies, then the limit inferior of this quotient 
is $2 \sin 36^{\circ } = 1.1756 \ldots $.

\medskip

{\bf{2010 Mathematics Subject Classification.}}
Primary: 52B10; %Three-dimensional polytopes
Secondary: 52A40%Inequalities and extremum problems

{\bf{Keywords and phrases.}}
Convex polyhedra, uniformity of the networks

\medskip

\section{Introduction}

We consider convex polyhedra inscribed in the unit sphere in the
3-dimen\-sion\-al Euclidean space, having $m$ faces, all of which are
triangles.
Let $\eta$ denote the quotient of the length of the maximal edge and the
length of the minimal edge.
We have $\eta = 1$ if and only if the polyhedron is regular with triangular
faces (in which case $m = 4$, $8$ or $20$).
Thus in this sense $\eta$ measures the deviation of the polyhedron from the
regular triangular ones.
Since for $m > 20$ there are no regular triangular polyhedra, we have for
large $m$ \ $\eta > 1$.
There arises the question, how closely regular can such a polyhedron be in
this sense, i.e.\ how can one construct such polyhedra, with $m \to \infty$,
such that the quotient $\eta$ remains as close to $1$ as possible.
We will prove the following.

%%%%%%%%%%%%%%%%%%%%%%%%%%%%%%%%%%%%%%%%%%%%%%%%%%%%%%%%%%%%%%%%%%%%%%%%%%%%

\begin{theo}
Let $e_m$ denote the infimum of $\eta = \,$length of maximal edge/length of
minimal edge, for all convex polyhedra inscribed in a unit sphere in the
$3$-dimensional Euclidean space, having $m$ faces, all of which are triangles.
Then we have $\liminf\limits_{m\to\infty} e_m = 2 \sin 36^\circ =
1.1756\ldots\ $.
\end{theo}
 
%%%%%%%%%%%%%%%%%%%%%%%%%%%%%%%%%%%%%%%%%%%%%%%%%%%%%%%%%%%%%%%%%%%%%%%%%%%%%%

The proof of the theorem consists of the proof of two inequalities:
$$
\liminf_{m \to \infty} e_m \geq 2 \sin 36^\circ \ \ \ \text{ and }
\ \ \ \liminf_{m \to \infty} e_m \leq 2 \sin  36^\circ.
$$
The first one is proved using the idea of L. Fejes T\'oth.
The second one has been known for some time.
Namely Clinton \cite{Cl1} constructed polyhedra, for which $m \to \infty$,
which he found to satisfy
(Clinton \cite{Cl2}, \cite{Cl3}) $\eta \leq 2 \sin 36^\circ$ (presumably on
the basis of calculations on computers).
They will be shortly discussed in Section~\ref{sec:2}.
Later Kitrick \cite{Ki} has given another construction which however seemed
to be handled much more simply when calculating~$\eta$.
His paper asserts that $\eta \leq 2 \sin 36^\circ$ holds for these polyhedra
as well, however it does not contain any hint about the mathematical proof of
$\eta \leq 2 \sin 36^\circ$ either.
Actually we will prove that Kitrick's polyhedra satisfy $\eta \leq 2 \sin
36^\circ$.

This question has its origin in architecture. Namely, building engineers want
to build large spherical domes, which consist of triangular elements (faces).
For them it is important, how uniform this inscribed triangular polyhedron is,
i.e., that the above $\eta $ should be as small as possible. This has partly
esthetic reasons, and also at actual construction one can manifacture rods of
equal length, and then only the joints at the vertices are produced differently.
There are also some other objectives at building spherical domes, usually
expressed by some other minimality conditions.
The author thanks T. Tarnai for having turned his attention to this question,
which was later repeated also by L. Fejes T\'oth, who proved the asymptotic 
lower estimate $2 \sin 36^{\circ }$ for $\eta $. 

The two inequalities will be contained in Propositions \ref{prop:1} and
\ref{prop:2}, and in fact they will be proved in a bit sharper form.

%%%%%%%%%%%%%%%%%%%%%%%%%%%%%%%%%%%%%%%%%%%%%%%%%%%%%%%%%%%%%%%%%%%%%%%%%%%%

\section{The lower estimate}
\label{sec:1}

\begin{proposition}
\label{prop:1}
Let a convex polyhedron inscribed in a unit sphere have $m$ faces, all of
which are triangles.
Then we have $\eta = \,$ length of maximal edge/
length of minimal edge $\geq 2 \sin 36^\circ \sqrt{1 - (d_k^2 / 4)} \geq \sin
36^\circ / \sin (30^\circ (m + 4)/m)$.
That is, we have (with $e_m$ as defined in the Theorem) $e_m \geq 2 \sin
36^\circ \sqrt{1 - (d_k^2 / 4)} \geq \sin 36^\circ / \sin (30^\circ (m +
4)/m)$, consequently $\liminf\limits_{m \to \infty} e_m \geq 2 \sin 36^\circ$.
Here $k = \frac12 m + 2$ and $d_k$ denotes the maximum of
$g(P_1, \dots, P_k) = \min\limits_{i\neq j} P_i P_j$, where $P_1, \dots, P_k$
is an arbitrary system of points on the unit sphere (and $P_i P_j$ means the
Euclidean length of the straight line segment $P_i P_j$).
\end{proposition}

The idea of the proof is due to L. Fejes T\'oth (oral communication).

%%%%%%%%%%%%%%%%%%%%%%%%%%%%%%%%%%%%%%%%%%%%%%%%%%%%%%%%%%%%%%%%%%%%%%%%%%%%

\begin{proof}
By \cite[p.\ 115]{FT} we have
$$
d_k^2 \leq 4 - \frac{1}{\sin^2\left(\frac{k}{k - 2}\right) \cdot 30^\circ} = 4
- \frac{1}{\sin^2\left(\frac{m + 4}{m} \cdot 30^\circ\right)}.
$$
Hence
$$
2\sin 36^\circ \sqrt{1 - \frac{d_k^2}{4}} \geq \sin 36^\circ \Bigm/ \sin
\left(\frac{m + 4}{m} \cdot 30^\circ\right).
$$
Thus it suffices to show the first inequality of the proposition.

By Euler's theorem on polyhedra $k = \frac{m}{2} + 2$ is the number of
vertices of the polyhedron (\cite[p.\ 114]{FT}), and there is a vertex to
which at most five edges are adjacent (\cite[p.\ 15 (6)]{FT}).

We see easily (by convexity) that a straight line segment connecting two
vertices of the polyhedron of minimal distance is an edge of the polyhedron.
Thus the minimal edge of the polyhedron has length $\leq d_k$.

Let us suppose first that the vertices of the polyhedron are on a closed
half-sphere.
Let us consider a supporting plane of the polyhedron separating (not strictly)
the polyhedron and the centre of the sphere, having a maximal distance from
the centre of the sphere.
Let this plane intersect the sphere in a circle $C$ of radius~$r$.
Then it is easily seen that the intersection of the polyhedron and this plane
contains the centre of the circle~$C$.
Thus the polyhedron has an edge (in the plane of the circle) of length $\geq 2
r/\sqrt{3}$.

We may suppose that the length of the minimal edge of the polyhedron is $\geq
\bigl(2r/\sqrt{3}\bigr)/2 \sin 36^\circ$.
Let us consider a half-sphere with the above circle $C$ as equator.
Let us project the spherical cap (on the unit sphere) bounded by $C$ that
contains the vertices of the polyhedron by rays perpendicular to the plane of
$C$ to the half-sphere above our circle $C$.
One sees easily that this mapping is distance-increasing (not strictly).
This is to be shown only for the ``vertical'' coordinates, and is a
consequence of the monotonity of $\sqrt{r^2 - a^2} - \sqrt{r^2 - b^2}$, as a
function of $r$, for $0 \leq a \leq b \leq r$.
Therefore by this projection the vertices of the polyhedron are carried over
to points $P_1',\dots, P_k'$ on the half-sphere, satisfying
$$
\min_{i \neq j} P_i' P_j' \geq \frac{2}{\sqrt{3}} r \bigm/ 2 \sin 36^\circ >
0.9822 r.
$$

%%%%%%%%%%%%%%%%%%%%%%%%%%%%%%%%%%%%%%%%%%%%%%%%%%%%%%%%%%%%%%%%%%%%%%%%%

\begin{lem}
Let $P_1',\dots, P_k'$, $k \geq 12$, be points on a half-sphere of radius $r$.
Then
$$
\min_{i \neq j} P_i' P_j' \leq \frac{24 \sqrt{6} + 11}{145} r < 0.9626 r.
$$
\end{lem}

%%%%%%%%%%%%%%%%%%%%%%%%%%%%%%%%%%%%%%%%%%%%%%%%%%%%%%%%%%%%%%%%%%%%%%%%%%

\begin{proof}
The spherical caps with centres $P_i'$ and (Euclidean) diameter
$\min\limits_{i \neq j} P_i' P_j'$ are non-overlapping and are on the
spherical surface $S$ of a segment of the sphere of radius $r$ containing the
half-sphere, of height $r + \frac12 \min\limits_{i \neq j} P_i'P_j'$.
The area of the surface $S$ is $2r^2 \pi\Bigl(r + \frac12 \min\limits_{i \neq
  j} P_i'P_j'\Bigr)$, and the areas of the circles are $2r^2 \pi\biggl(r -
\sqrt{r^2 - \Bigl(\frac12 \min\limits_{i \neq j} P_i'P_j'\Bigr)^2}\biggr)$.
Their quotient is at least $12$, whence the lemma follows.
\end{proof}

%%%%%%%%%%%%%%%%%%%%%%%%%%%%%%%%%%%%%%%%%%%%%%%%%%%%%%%%%%%%%%%%%%%%%%%%%%%%

We continue with the proof of the proposition.
By \cite[pp. 114-115]{FT} we have for $k \leq 12$
$$
2 \sin 36^\circ \sqrt{1 - \frac{d_k^2}{4}} \leq 2 \sin 36^\circ \sqrt{1 -
  \frac{d_{12}^2}{4}} = 1 \ \ \ (\leq \eta).
$$
Thus we can suppose $k > 12$.
Then the above considerations, together with the lemma, show that the vertices
are not on a closed half-sphere, unless $\eta \geq 2 \sin 36^\circ$.

Thus we may suppose that the centre of our unit sphere is in the interior of
the polyhedron.
We project the edges of the polyhedron on the surface of the sphere.
Thus we obtain a tiling on the sphere, consisting of spherical triangles
corresponding to the faces of the polyhedron.
There is a vertex to which at most five edges are adjacent, thus there is a
spherical triangle with an angle $\geq 72^\circ$.

Now let us suppose that
$$
\eta < \mu = 2 \sin 36^\circ \sqrt{1 - \frac{d_k^2}{4}}.
$$
Then the longest edge has a length $\leq \eta d_k < 2 \sin 36^\circ \sqrt{1 -
  (d_k^2/4)} \cdot d_k \leq 2\sin 36^\circ$.
Thus each side of every spherical triangle in the tiling has length $<
\frac{\pi}{2}$.

Let $ABC$ be a spherical triangle in the tiling, with maximal angle $\gamma =
\sphericalangle ACB \geq 72^\circ$.
Let $\widetilde{AB} = c$, $\widetilde{BC} = a$, $\widetilde{CA} = b$
($\widetilde{\phantom{AB}}$ meaning the spherical length of the respective
side), $a \leq b$.
If $\gamma$ and $a$ are fixed and $b$ is decreased to $b = a$, $c$ will
decrease too, to $c = c_0$.

For $a = b$ the quotient $AB/BC$ of the (Euclidean) lengths is $\sin
\frac{c}{2} \bigm/ \sin \frac{a}{2}$.
By spherical trigonometry $\sin \frac{c_0}{2} = \sin \frac{\gamma}{2} \cdot
\sin a$, so $\sin \frac{c_0}{2} \bigm/ \sin \frac{a}{2} = 2 \sin
\frac{\gamma}{2} \cdot \cos \frac{a}{2}$.
So for our triangle $ABC$ we have $\sin \frac{c}{2} \bigm/ \sin \frac{a}{2}
\geq 2 \sin 36^\circ \cdot \cos \frac{a}{2}$.

Since the minimal edge of the polyhedron has (Euclidean) length $\leq d_k$, we
have $c_0 \leq c \leq 2 \arcsin \left(\frac12 \eta d_k\right)$.

We have
$$
\eta \geq \sin \frac{c}{2} \Bigm/ \sin \frac{a}{2} \geq 2 \sin 36^\circ \cdot
\cos \frac{a}{2} \geq 2 \sin 36^\circ \cdot \cos \frac{a_1}{2},
$$
where $a_1 \leq \frac{\pi}{2}$ is defined by $\frac12 \eta d_k = \sin 36^\circ
\cdot \sin a_1$.
(As above stated, we have $\frac12 \eta d_k \leq \sin 36^\circ$.)
Namely
$$
\sin a = \sin \frac{c_0}{2} \Bigm/ \sin \frac{\gamma}{2} \leq \frac{\eta
  d_k}{2} \Bigm/ \sin 36^\circ = \sin a_1.
$$
Thus
\begin{align*}
\eta \geq 2 \sin 36^\circ \cdot \cos \frac{a_1}{2} &= 2 \sin 36^\circ
\sqrt{\frac12 \Bigl(1 + \sqrt{1 - \sin^2 a_1}\Bigr)} \\
&= 2\sin 36^\circ \sqrt{\frac12 \left(1 + \sqrt{1 - \left(\frac{\eta
      d_k}{2\sin 36^\circ}\right)^2}\right)} =: f(\eta),
\end{align*}
where the last equality holds by definition.

Let us take into account that for $1 \leq \eta \leq 2 \sin 36^\circ$ $f(\eta)$
is decreasing and $\mu = f(\mu)$ (this holds, since $d_k < \sqrt{2}$).
Thus we see that $\eta < \mu$ implies $\eta < f(\eta)$, which is a
contradiction.
Hence the indirect assumption $\eta < \mu$ cannot hold, which proves the
proposition.
\end{proof}

%%%%%%%%%%%%%%%%%%%%%%%%%%%%%%%%%%%%%%%%%%%%%%%%%%%%%%%%%%%%%%%%%%%%%%%%

\begin{rema}
The best upper estimate of $d_k$ known at present is due to Robinson \cite{Ro}.
We have used the estimate of Fejes T\'oth \cite{FT} because of its simplicity
and because it is already asymptotically sharp, \cite[p.\ 114]{FT}.
Thus the utilisation of the results of Robinson \cite{Ro} yields only a slight
improvement for large~$m$.
\end{rema}

%%%%%%%%%%%%%%%%%%%%%%%%%%%%%%%%%%%%%%%%%%%%%%%%%%%%%%%%%%%%%%%%%%%%%%%%%%%

\section{The upper estimate}
\label{sec:2}

\subsection*{A) Clinton's polyhedra}

We first sketch the construction of Clinton \cite{Cl1}.
He considered first a regular icosahedron inscribed in a unit sphere and then
projected it radially (i.e.\ from the centre of the sphere) to the surface of
the sphere.
In the spherical tiling thus obtained each spherical edge was divided into $n$
equal parts ($n \geq 1$ integer).
The dividing points are then carried back by radial projection to the surface
of the inscribed icosahedron.
In this way on a face of the icosahedron we obtain points $A_1, \dots, A_{n -
  1}$; $B_1, \dots, B_{n - 1}$; $C_1, \dots, C_{n - 1}$, these points lying on
the sides $a$, $b$, $c$ of the face, respectively, in this order and following
each other in the positive orientation.

Now draw through each $A_i$ parallels to $b$ and $c$, through each $B_j$
parallels to $c$ and $a$ and through each $C_k$ parallels to $a$ and $b$.

If we replace the points $A_1, \dots, C_{n - 1}$ by points $A_1', \dots, C_{n
  - 1}'$ which lie on the same sides in the same order but divide the sides
into $n$ equal parts, then drawing the parallels similarly we obtain a
subdivision of the face of the icosahedron into $n^2$ regular triangles.
Let $P'$ be one of the vertices of one of these regular triangles, which is
not a vertex of the face of the icosahedron; it lies on three of these
parallels through some points $A_i'$, $B_j'$, $C_k'$ (counting now the sides
of the face to these parallels, too). 
Consider now the corresponding parallels through the respective points $A_i$,
$B_j$, $C_k$.
These will not be concurrent in general but will bound a regular triangle
$T(P')$.
We denote the centre of this triangle by~$P$.
If $P'$ is a vertex of the face of the icosahedron, let $P = P'$.
In any case we denote the radial projection of $P$ to the surface of the
sphere by $Q$.

The vertices of Clinton's polyhedron will be all the points $Q$ (derived from
all the points $P'$ on all the faces of the icosahedron).
Two vertices $Q$ and $\overline{Q}$ will be connected by an edge if and only
if the points $P'$ and $\overline{P'}$ corresponding to them lie on the same
face of the icosahedron and are connected by an edge (in the subdivision of
the face into the $n^2$ regular triangles).
The number of faces of Clinton's polyhedron is $m = 20 n^2$.

When trying to check that for Clinton's polyhedron for $n \to \infty$ there
holds $\eta \to 2 \sin 36^\circ$, we have to determine the edge lengths of
this polyhedron asymptotically.
This leads to the determination of the maximum and the minimum of a certain
function of two variables.
However this function is given in a rather involved way and its discussion
could be accomplished probably only by lengthy calculations and possibly by
using computers.
Thus for Clinton's polyhedra the relation $\eta \to 2 \sin 36^\circ$ is still
only a conjecture.

Clinton \cite{Cl1} has given also another construction which he also found to
satisfy $\eta \to 2 \sin 36^\circ$.
The only difference is that now he derives from the points $P'$, instead of
the points $Q$ derived above, rather the points $R$ which are defined in the
following way.
For $P'$ a vertex of a face of the icosahedron let $R = P'$.
Otherwise consider the above defined regular triangle $T(P')$.
Projecting its vertices radially on the surface of the sphere we obtain the
vertices of another planar triangle $T_1(P')$.
The incentre of $T_1(P')$ will be projected radially on the surface of the
sphere, to obtain the point $R$.
The vertices of this second polyhedron of Clinton will be all the points $R$,
and two vertices $R$ and $\overline{R}$ will be connected by an edge under the
same condition as above for $Q$ and $\overline{Q}$.
However for this second polyhedron the (asymptotic) calculations seem to be
even more complicated, thus the relation $\eta \to 2 \sin 36^\circ$ is only a
conjecture, in this case as well.

%%%%%%%%%%%%%%%%%%%%%%%%%%%%%%%%%%%%%%%%%%%%%%%%%%%%%%%%%%%%%%%%%%%%%%%%%%%

\subsection*{B) Kitrick's polyhedra}

Now we recall the definition of Kitrick's polyhedra, \cite{Ki}.
Consider a regular dodecahedron inscribed in a unit sphere.
Project its faces radially to the surface of the sphere.
Thus we obtain a tiling on the sphere consisting of regular spherical pentagons.
Decompose each of these pentagons to five isosceles spherical
triangles of common vertex the centre of the pentagon and bases the sides of
the pentagon.
The angles of these triangles are $72^\circ$, $60^\circ$, $60^\circ$.
Each of these triangles will be subdivided into $n^2$ smaller triangles in the
same way.

Let us suppose the axis of symmetry of one of these triangles is the equator
of the sphere.
Now divide the equal sides of our isosceles triangle into $n$ equal parts ($n
\geq 1$ integer).
Through the dividing points draw (the portions inside the triangle of the)
meridians and parallels.
These divide our spherical triangle into $2n$ triangles and $n^2 - n$
symmetric quadrangles (which are bounded by arcs of great and small circles).
Then take every second point of division on each meridian, beginning from the
ones on the equal sides of our isosceles triangle and connect them by
(shorter) arcs of large circles, like in the schematic Fig.~1.

%%%%%%%%%%%%%%%%%%%%%%%%%%%%%%%%%%%%%%%%%%%%%%%%%%%%%%%%%%%%%%%%%%%%%%%%

\begin{center}

\includegraphics[height=100mm]{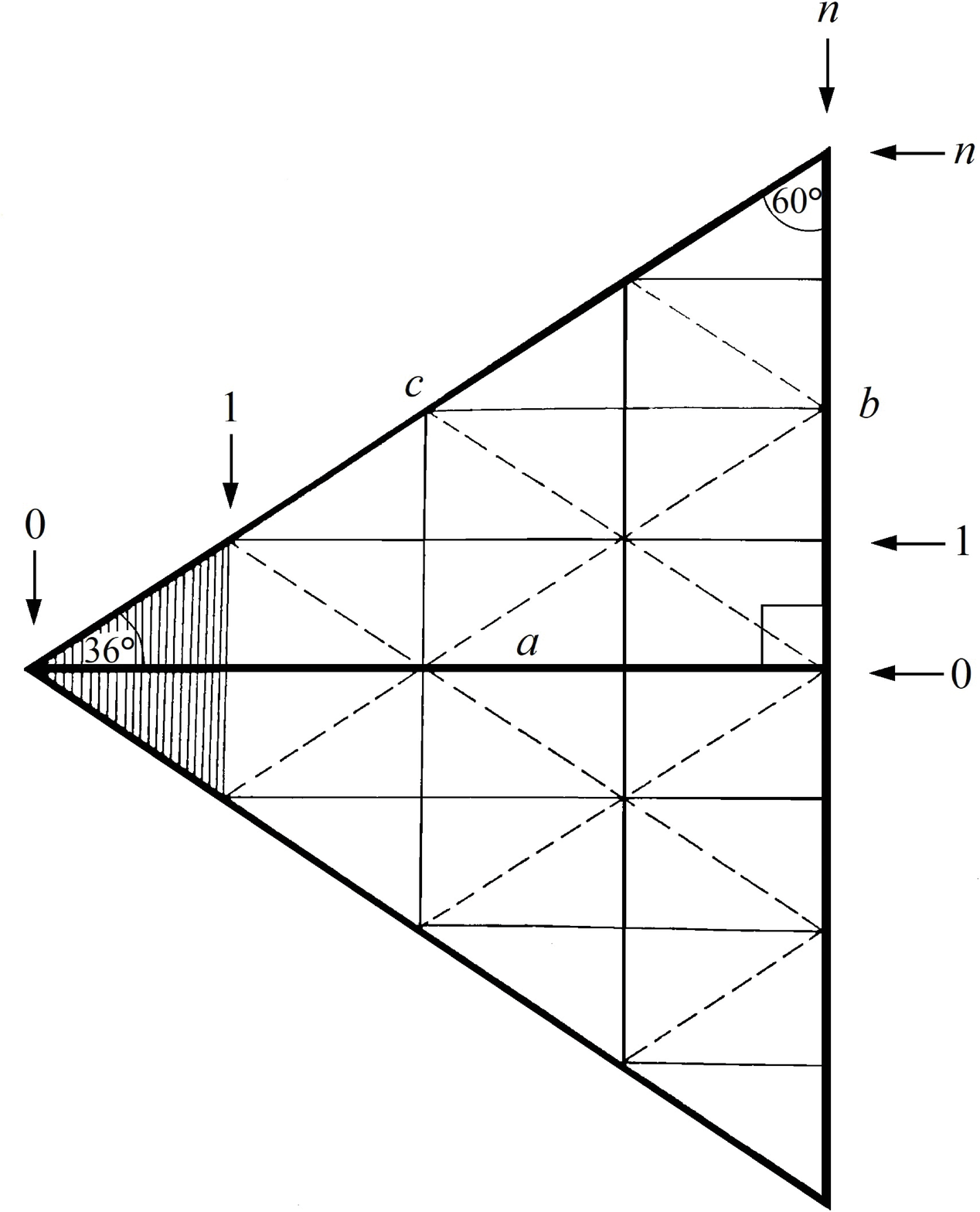}

\medskip
{\small {Figure 1}}
\end{center}

%%%%%%%%%%%%%%%%%%%%%%%%%%%%%%%%%%%%%%%%%%%%%%%%%%%%%%%%%%%%%%%%%%%%%%%%

Proceeding similarly for each of the spherical triangles of angles $72^\circ$,
$60^\circ$, $60^\circ$ in our tiling we obtain a new tiling consisting of
spherical triangles.
Replacing each spherical triangle of this new tiling by a planar triangle
having the same vertices we obtain a polyhedron.
This is Kitrick's polyhedron.
The number of its faces is $m = 60 n^2$.

We have not shown yet that this polyhedron is convex.
However this will follow from $\eta \leq 2 \sin 36^\circ < \sqrt{2}$
(cf.\ Proposition \ref{prop:2}).
Namely by elementary geometry
$\eta < \sqrt{2}$ implies that each face 
of the polyhedron is an acute triangle, 
which in turn implies that the
polyhedron, which is inscribed in a sphere, is convex.

%%%%%%%%%%%%%%%%%%%%%%%%%%%%%%%%%%%%%%%%%%%%%%%%%%%%%%%%%%%%%%%%%%%%%%%%%%%

\begin{proposition}
\label{prop:2}
We have for Kitrick's polyhedra, defined above, for every integer $n \geq 1$
(and with $e_m$ as defined in the Theorem) $e_{60 n^2} \leq \eta =$ length of
maximal edge / length of minimal edge $= 2\sin 36^\circ \cos
\bigl[\bigl(\arccos (\cot 36^\circ / \sqrt{3})\bigr) / 2n\bigr] = 2\sin
36^\circ \cos \bigl((37.37^\circ  \ldots)/2 n\bigr) < 2 \sin 36^\circ$.
Consequently we have $\liminf\limits_{m \to \infty} e_m \leq 2 \sin 36^\circ$.

More exactly the maximal (minimal) edges of these polyhedra are the bases
(equal sides) of the faces (corresponding to the shaded spherical triangle in
Fig.~1) which are isosceles triangles and contain the vertices with angle
$72^\circ$ of our subdivided spherical triangles.
\end{proposition}

%%%%%%%%%%%%%%%%%%%%%%%%%%%%%%%%%%%%%%%%%%%%%%%%%%%%%%%%%%%%%%%%%%%%%%%%%%%%%

\begin{proof}
1) First we show $\lim\limits_{n \to \infty} \eta = 2 \sin 36^\circ$.

\smallskip
1.1) Consider the isosceles spherical triangle in Fig.~1.
It is cut by the equator in two right spherical triangles of angles
$36^\circ$, $60^\circ$, $90^\circ$.
The sides of these triangles opposite to these angles will be denoted by $a,
b, c$.
We have by spherical trigonometry $\cos a = 1/(2\sin 36^\circ)$, $\cos b =
\cos 36^\circ/(\sqrt{3}/2)$, $\cos c = \cot 36^\circ / \sqrt{3}$.
Because of symmetry it suffices to consider the edges of the subdivision lying
in the upper right triangle (which will be denoted by $T$) or intersecting the
side~$a$.

We number the meridians in Fig.~1 consecutively by the numbers $0, \dots, n$,
the $0$th one being the (degenerate) one passing through the vertex with angle
$36^\circ$, the $n$th one being the one containing the side $b$.
Similarly we number the parallels in Fig.~1 meeting the upper right triangle
$T$ consecutively by the numbers $0, \dots, n$ the $0$th one being the
equator, the $n$th one being the (degenerate) one passing through the vertex
with angle $60^\circ$.
Let us denote the point of intersection of the $i$th meridian and the $j$th
parallel (in $T$) by $P_{i,j}$ ($0 \leq j \leq i \leq n$).

There are two kinds of edges: those lying on meridians and those which are
diagonals of the symmetric quadrangles (bounded by arcs of great and small
circles) in Fig.~1.
(The edges on the side $c$ will be counted to the second mentioned edges.)
These two kinds of edges will be called meridional and diagonal edges,
respectively.
We will show that among the edges of both kinds asymptotically the maximal and
minimal ones are those corresponding to the Euclidean edges mentioned in
Proposition \ref{prop:2}.
That is, if we express the lengths of the edges incident to $P_{i,j}$ in the
form $f(i/n, j/n) \cdot (1/n) + $ terms of order smaller than $1/n$, then for
both cases the function $f(i/n, j/n)$ will only assume values between the two
values assumed for the edges of the shaded spherical triangle.
In view of the fact that (asymptotically) the lengths of the edges are
determined by the directional derivatives of a continuously differentiable
function -- via the mean value theorem -- this will imply the equality
$\lim\limits_{n \to \infty} \eta = 2 \sin 36^\circ$ asserted above.

\smallskip
1.2) Our spherical triangle $T$ is cut by the $i$th meridian in two parts.
One of the parts is a right triangle.
Its side lying on the side $c$ of $T$ has length $(i/n)c$, and we denote the
lengths of its sides on the side $a$ of $T$ and on the $i$th meridian by $a_i$
and $b_i$, respectively.
We have $\tan a_i = \tan ((i/n)c)\cdot \cos 36^\circ$, $\sin b_i =
\sin((i/n)c)\cdot \sin 36^\circ$.
Consequently
\begin{align*}
a_{i + 1} - a_i &\sim \frac{\cos 36^\circ}{\cos^2\left(\frac{i}{n} c\right) +
  \sin^2 \left(\frac{i}{n} c\right) \cdot \cos^2 36^\circ} \cdot
\frac{c}{n},\\
b_{j + 1} - b_j &\sim \frac{\cos \left(\frac{j}{n} c\right)\cdot \sin 36^\circ}
{\sqrt{\cos^2\left(\frac{j}{n} c\right) + \sin^2 \left(\frac{j}{n} c\right)
    \cdot \cos^2 36^\circ}} \cdot \frac{c}{n}.
\end{align*}
The point $P_{i,j}$ lies on the side of length $b_i$ of the triangular part of
$T$ at a distance $b_j$ from the equator.

The sides of the shaded spherical triangle have lengths $c/n$ and $2 b_1 \sim
2 \sin 36^\circ \cdot c/n$.
Thus we have to show that the edge lengths are asymptotically between $c/n$
and $2 \sin 36^\circ \cdot c/n$.

\smallskip
1.3) We first consider the meridional edges.
These are of the form $P_{i,j - 1} P_{i,j + 1}$ ($P_{i, -1}$ being the mirror
image of $P_{i,1}$ w.r.t.\ the equator).
The squares of their lengths are
$$
\sim (b_{j + 1} - b_{j - 1})^2 \sim \frac{4\sin^2 36^\circ}{1 + \tan^2
  \left(\frac{j}{n} c\right) \cdot \cos^2 36^\circ} \cdot
\left(\frac{c}{n}\right)^2.
$$
Here the coefficient of $(c/n)^2$ is a monotonically decreasing function of
$(j/n)c$, hence
$$
4 \sin^2 36^\circ \geq \frac{4\sin^2 36^\circ}{1 + \tan^2 \left(\frac{j}{n}
  c\right)\cdot \cos^2 36^\circ} \geq \frac{4\sin^2 36^\circ}{1 + \tan^2 c
  \cdot \cos^2 36^\circ}.
$$
Thus the upper estimate of the length of the meridional edges is proved.
For the lower estimate observe that because of $\cos c = \cot 36^\circ/
\sqrt{3}$ we have
$4 \sin^2 36^\circ / (1 + \tan^2 c \cdot \cos^2 36^\circ) = 1$.
Thus both asymptotic inequalities are proved in the case of the meridional
edges.

Now we turn to the diagonal edges.
These have the form $P_{i,j} P_{i + 1, j + 1}$ or $P_{i + 1,j} P_{i,j + 1}$.
However on account of symmetry the lengths of both are the same and their
squares are $\sim \cos^2 b_j \cdot(a_{i + 1} - a_i)^2 + (b_{j + 1} - b_j)^2
\sim$
$$
\sim \left[ \frac{\cos^2 36^\circ \cdot \cos^2 \left(\frac{j}{n}
    c\right)}{\left[1 - \sin^2 \left(\frac{i}{n}c\right) \cdot \sin^2
      36^\circ\right]^2} + \frac{\sin^2 36^\circ}{1 + \tan^2\left(\frac{j}{n}
    c\right) \cdot \cos^2 36^\circ}\right] \cdot \left(\frac{c}{n}\right)^2 .
$$
Here again the coefficient of $(c/n)^2$ is a monotonically decreasing function
of $(\bar  j / n)c$.
Recalling that $0 \leq j \leq i$ $(\leq n)$ we have to show that substituting
$(\bar j/n)c = (i/n)c$,
or $(j/n)c = 0$, resp., the coefficient of $(c/n)^2$ will be at least $1$, or
at most $4\sin^2 36^\circ$, resp.
However substituting $(j/n)c = (i/n)c$ means calculating asymptotically the
lengths of the (spherical) edges $P_{i,i} P_{i + 1, i + 1}$, which are by the
very definition of Kitrick's polyhedra equal to $c/n$.
Thus the lower estimate of the length of the diagonal edges is proved.

Now we turn to the upper estimate for the diagonal edges.
Substituting $(j/n)c = 0$ in the coefficient of $(c/n)^2$ gives
$$
\frac{\cos^2 36^\circ}{\left[1 - \sin^2 \left(\frac{i}{n} c\right)\cdot \sin^2
    36^\circ\right]^2} + \sin^2 36^\circ.
$$
This is a monotonically increasing function of $(i/n)c$ (in the interval $[0,
  c]$), thus its largest value is attained for $(i/n)c = c$, and we have to
show it is $\leq 4 \sin^2 36^\circ$.

Taking into consideration that $\cos c = \cot 36^\circ/\sqrt{3}$ this last
inequality reduces to the inequality $3/4 \leq 4 \sin^2 36^\circ \cdot \cos^2
36^\circ = \sin^2 72^\circ$, which is in fact valid.
Thus both asymptotic inequalities are proved for the diagonal edges as well as
for the meridional edges.
As stated above this proves $\lim\limits_{n \to \infty} \eta = 2 \sin 36^\circ$.

\medskip
2) Now we will follow the above asymptotic calculations by exact calculations
valid for each $n \geq 1$ and will show that the edges $P_{1, - 1} P_{1,1} /
P_{0, 0} P_{1,1}$ of the face corresponding to the shaded spherical triangle
are actually the maximal and minimal edges of Kitrick's polyhedron.
The quotient $P_{1, - 1} P_{1,1}/P_{0,0} P_{1,1}$ of their Euclidean lengths
is $\sin b_1/\sin (c/2n) = 2 \sin 36^\circ \cdot \cos (c/2n)$.

\smallskip
2.1) We will consider the meridional edges $P_{i,j - 1} P_{i,j + 1}$ and
instead of all diagonal edges all segments $P_{i,j} P_{i + 1, j + 1}$ (these
have equal lengths on account of symmetry).
We will show that for given $i$ the lengths of $P_{i,j - 1} P_{i,j + 1}$
($P_{i,j} P_{i + 1, j + 1}$, resp.)
are actually -- i.e.\ not only asymptotically -- decreasing functions of~$j$.

First we consider the edges $P_{i,j - 1} P_{i,j + 1}$.
Let us denote their arc measures (i.e.\ the smaller angles subtended by them
at the centre of the sphere) by $\widetilde{P_{i, j - 1} P_{i,j + 1}}$.
(A similar notation will be applied for other edges and segments as well.)
We will show the decreasing of these arc measures.
We have $0 \leq j \leq i - 1$.
Now consider $j$ as a real variable in the above given interval, $P_{i, j}$
meaning the point of $T$ on the $i$th meridian, having a spherical distance
$b_j = \arcsin (\sin((j/n)c) \cdot \sin 36^\circ)$ from the equator.
We show $\widetilde{P_{i,j - 1} P_{i,j + 1}}$ is decreasing in that interval.
In fact
$$
\frac{d}{dj} \widetilde{P_{i,j - 1} P_{i,j + 1}} = \frac{\sin 36^\circ}{\sqrt{1
    + \tan^2\left(\frac{j + 1}{n} c\right) \cdot \cos^2 36^\circ}} -
\frac{\sin 36^\circ}{\sqrt{1 + \tan^2\left(\frac{j - 1}{n} c\right) \cdot
    \cos^2 36^\circ}} ,
$$
which is non-positive because of $\bigl|((j + 1)/n)c\bigr| \geq \bigl|((j -
1)/n)c\bigr|$.
Evidently the same considerations show that for $0 \leq j \leq i - 1$, we have
that
$\widetilde{P_{i,j} P_{i,j + 1}}$ is a decreasing function of~$j$.

Now we turn to the segments $P_{i,j} P_{i + 1,  j + 1}$.
We show $P_{i,j} P_{i + 1, j + 1} < P_{i,j} P_{i + 1, j - 1} = P_{i,j - 1}
P_{i + 1,j}$ $(1 \leq j \leq i)$.
In fact the points $P$ on the unit sphere satisfying $PP_{i + 1, j + 1} =
PP_{i + 1, j - 1}$ form a great circle perpendicular to the $i$th meridian and
meeting the $i$th meridian at midpoint $P^*$ of its smaller arc $P_{i + 1, j -
  1} P_{i + 1, j + 1}$.
However because of $P_{i + 1, j - 1} P_{i + 1, j} > P_{i + 1, j} P_{i + 1, j +
  1}$ the point $P^*$ lies on the arc $P_{i + 1, j - 1} P_{i + 1, j}$.

Hence by projecting on the plane spanned by the $(i + 1)$st meridian we see
the entire $j$th parallel (thus also $P_{i,j}$) lies nearer to $P_{i + 1, j +
  1}$ than to $P_{i + 1, j - 1}$.

\smallskip
2.2) Now we prove the upper and lower estimates for the meridional edges.
The upper estimate follows from the fact that $\widetilde{P_{i,j - 1} P_{i,j +
    1}}$ is a monotonically decreasing function of $j$, therefore
$\widetilde{P_{i,j - 1} P_{i,j + 1}} \leq \widetilde{P_{i, -1} P_{i, 1}} =
\widetilde{P_{1, - 1} P_{1,1}}$.

For the lower estimate observe that by the same monotonity property
$\widetilde{P_{i,j - 1} P_{i, j + 1}} = \widetilde{P_{n, j - 1} P_{n, j + 1}}
\geq \widetilde{P_{n, n - 2} P_{n, n}}$.
Now consider the spherical triangle $P_{n, n}$ $P_{n, n - 2}$ $P_{n - 1, n -
  1}$.
Its angle at $P_{n, n}$ is $60^\circ$, and we have $\widetilde{P_{n, n} P_{n -
    1, n - 1}} = c/n$.
Since $\widetilde{P_{i,j} P_{i + 1, j + 1}}$ is a monotonically decreasing
function of $j$, we have
$\widetilde{P_{n,n} P_{n - 1, n - 1}} < \widetilde{P_{n, n - 1} P_{n - 1, n -
    2}} = \widetilde{P_{n, n - 2} P_{n - 1, n - 1}}$.
Since in spherical triangles larger sides are opposite to larger angles, we have
$\sphericalangle P_{n, n} P_{n, n - 2} P_{n - 1, n - 1} < 60^\circ$.
Since the sum of the angles of our spherical triangle is $> 180^\circ$, this
implies $\sphericalangle P_{n, n} P_{n - 1, n - 1} P_{n, n _ 2} > 60^\circ$.
This implies in turn that the side $\widetilde{P_{n,n} P_{n, n - 2}}$ is
greater than the side $\widetilde{P_{n, n} P_{n - 1, n - 1}} = c/n$.
Thus we have
$\widetilde{P_{i, j - 1} P_{i, j + 1}} \geq \widetilde{P_{n, n - 2} P_{n, n}} >
c/n$.

\smallskip
2.3) Last we show the upper and lower estimates for the diagonal edges.
The lower estimate follows from the fact that $\widetilde{P_{i,j} P_{i + 1, j +
    1}}$ is a monotonically decreasing function of $j$, hence
$\widetilde{P_{i,j} P_{i + 1, j + 1}} \geq \widetilde{P_{i,i} P_{i + 1, i + 1}}
= c/n$.

For the upper estimate observe that by the same monotonity property
$\widetilde{P_{i,j} P_{i + 1, j + 1}} \leq \widetilde{P_{i, 0} P_{i + 1,1}}$.
We have from the right spherical triangle\break 
$P_{i,0} P_{i + 1,0} P_{i + 1,1} \cos \widetilde{P_{i,0} P_{i + 1, 1}} = \cos
\widetilde{P_{i + 1,0} P_{i + 1,1}} \cdot \cos{P_{i,0} P_{i + 1,0}} = \cos
\widetilde{P_{1,0} P_{1,1}} \cdot \cos \widetilde{P_{i, 0} P_{i + 1,0}}$.
We will show that for $0 \leq i \leq n - 1$, we have that 
$\widetilde{P_{1,0} P_{i + 1, 1}}$
is an increasing function of $i$.
In view of the last formula it suffices to show the increasing of
$\widetilde{P_{i,0} P_{i + 1,0}}$.
For this consider $i$ as a real variable in the given interval, $P_i$ meaning
the point of $T$ on the equator, having a spherical distance $a_i = \arctan
(\tan((i/n)c) \cdot \cos 36^\circ)$ from $P_{0,0}$.
We have
$$
\frac{d}{di} \widetilde{P_{i,0} P_{i + 1, 0}} = \frac{\cos 36^\circ}{1 - \sin^2
  \left(\frac{i + 1}{n} c\right) \cdot \sin^2 36^\circ} - \frac{\cos
  36^\circ}{1 - \sin^2\left(\frac{i}{n} c\right) \cdot \sin^2 36^\circ} > 0,
$$
hence in fact $\widetilde{P_{i,0} P_{i + 1,0}}$ is increasing.

Therefore also $\widetilde{P_{i,0} P_{i + 1,1}}$ is increasing, thus
$\widetilde{P_{i,0} P_{i + 1, 1}} \leq \widetilde{P_{n - 1, 0} P_{n, 1}}$.
It remains to show $\widetilde{P_{n - 1, 0} P_{n, 1}} \leq \widetilde{P_{1, -1}
  P_{1,1}}$.

Consider for a positive integer $l$ the $nl$th Kitrick polyhedron (i.e., the
side $c$ is divided into $nl$ equal parts).
The analogues of the points $P_{i,j}$ for this new polyhedron will be denoted
by $P'_{i,j}$ $(0 \leq j \leq i \leq nl)$.
We have $P_{n - 1, 0} = P'_{(n - 1)1, 0}$, $P_{n, 1} = P'_{nl, l}$.
In view of the triangle inequality we have $\widetilde{P_{n - 1, 0} P_{n, 1}} =
\widetilde{P'_{(n - 1) 1, 0} P'_{nl, l}} \leq \sum\limits_{r = 0}^{l - 1}
\widetilde{P'_{(n - 1) l + r, r} P'_{(n - 1)l + r + 1, r + 1}}$.

Because of the monotonity of the length of the diagonal edges (shown in part
2.1) of the proof)
$\widetilde{P'_{(n - 1) l + r, r} P'_{(n - 1) l + r + 1, r + 1}} \leq
\widetilde{P'_{(n - 1) l + r, 0} P'_{(n - 1) l + r + 1, 1}}$.
However these last edges have monotonically increasing lengths, as shown in
the last paragraph, hence
$\widetilde{P'_{(n - 1)l + r, 0} P'_{(n - 1) l + r + 1, 1}} \leq
\overline{P'_{nl - 1, 0} P'_{nl, 1}}$.
From the last three inequalities we obtain $\widetilde{P_{n - 1, 0} P_{n, 1}}
\leq l \cdot \widetilde{P'_{nl - 1, 0} P'_{nl, 1}}$.

As calculated in part 1.3) of the proof for $nl$ large we have asymptotically
$$
\widetilde{P'_{nl - 1, 0} P'_{nl, 1}} \sim \sqrt{\frac{\cos^2 36^\circ}{[1 -
      \sin^2 c \cdot \sin^2 36^\circ]^2} + \sin^2 36^\circ} \cdot
\frac{c}{nl}.
$$
Using this asymptotical equality we pass in the last inequality to the limit
$l \to \infty$.
Thus we obtain the inequality
$$
\widetilde{P_{n - 1, 0} P_{n, 1}} \leq \sqrt{\frac{\cos^2 36^\circ}{[1 - \sin^2
      c \cdot \sin^2 36^\circ]^2} + \sin^2 36^\circ} \cdot \frac{c}{n} = :
2\lambda \cdot \frac{c}{n}
$$
(where the last equality holds by definition) valid for our actual $n$ (not
only asymptotically).

It has remained to show that
$\widetilde{P_{n - 1, 0} P_{n, 1}} \leq \widetilde{P_{1, -1} P_{1, 1}}$.
In virtue of the inequality showed in the last paragraph it suffices to show
$2\lambda \cdot c/n \leq \widetilde{P_{1, -1} P_{1,1}} = 2b_1 = 2 \arcsin
(\sin(c/n) \cdot \sin 36^\circ)$, or equivalently (because of $\lambda \cdot
c/n \leq \lambda c < 90^\circ$)
$\sin(\lambda \cdot c/n) \bigm/ \sin (c/n) \leq \sin 36^\circ$.
However in view of $\lambda$ $(=0.5488\ldots) < 1 $ we have that
$\sin(\lambda x) / \sin x$ is monotonically increasing for $0 < x \leq
90^\circ$, further $c/n \leq c = 37^\circ 22'\ldots$; thus $\sin(\lambda \cdot
c/n) \bigm/ \sin (c/n) \leq \sin(\lambda c) / \sin c = 0.5773\ldots < \sin
36^\circ = 0.5878\ldots$.
This ends the proof of the upper estimate of the diagonal edges.
Since all the other estimates have been already proved above, this completes
the proof of the proposition.
\end{proof}

%%%%%%%%%%%%%%%%%%%%%%%%%%%%%%%%%%%%%%%%%%%%%%%%%%%%%%%%%%%%%%%%%%%%%%%%%%


\begin{thebibliography}{99}

\bibitem{Cl1} Clinton, J. D.,
NASA contractor report (NASA CR-1734), Advanced structural geometry studies,
I. Polyhedral subdivision concepts for structural applications, NASA,
Washington, D. C., 1971.

\bibitem{Cl2} Clinton, J. D.,
Unpublished manuscript, 1978.

\bibitem{Cl3} Clinton, J. D.,
$(p, q +)_{b, c}^\eta$, Preprint, Kean College, Union, N. J. USA, 1980.

\bibitem{FT} Fejes T\'oth, L.,
{\it Lagerungen in der Ebene, auf der Kugel und im Raum}, Springer-Verlag,
Berlin--G\"ottingen--Heidelberg, 1953.

\bibitem{Ki} Kitrick, C. J., Geodesic domes, {\it Struct. Topology} {\bf 11}
  (1985), 15--20.

\bibitem{MaTa} Makai, E. Jr., Tarnai, T.,
Morphology of spherical grids, {\it Acta Techn. Acad. Sci. Hungar.} {\bf 83}
(3-4), (1976), 247--283.

\bibitem{Ro} Robinson, R. M.,
Arrangement of 24 points on a sphere,
{\it Math. Ann.} {\bf 144} (1961), 17--48.

\end{thebibliography}
\end{document}